\documentclass[11pt,epsf]{article}

\usepackage{enumerate}
\usepackage{mathrsfs}
\usepackage{amsmath,color}
\usepackage{latexsym}
\usepackage{amssymb}
\usepackage{amsthm}
\usepackage{amsfonts}
\usepackage{dsfont}
\usepackage{hyperref}
\usepackage{authblk}

\newcommand{\bd}{\mathbf{D}}

\newcommand{\1}{{\bf 1}}

\newcommand{\lp}{\left(}
\newcommand{\rp}{\right)}
\newcommand{\lc}{\left[}
\newcommand{\rc}{\right]}

\newcommand{\ga}{\gamma}

\newcommand{\beq}{\begin{equation}}
\newcommand{\eeq}{\end{equation}}
\newcommand{\bea}{\begin{eqnarray}}
\newcommand{\eea}{\end{eqnarray}}
\newcommand{\beas}{\begin{eqnarray*}}
\newcommand{\eeas}{\end{eqnarray*}}

\def\me{{\mathbb  E}}

\def\mr{{\mathbb  R}}

\newcommand{\cac}{{\mathcal C}}

\newcommand{\ch}{{\mathcal H}}

\newcommand{\cn}{{\mathcal N}}

\newcommand{\IP}{{\mathbb P}}
\newcommand{\IE}{{\mathbb E}}
\newcommand{\IR}{{\mathbb R}}

\newcommand{\D}{{\mathbb D}}
\newcommand{\EE}{{\mathbb E}}

\newcommand{\R}{{\mathbb R}}

\newtheorem{theorem}{Theorem}[section]

\newtheorem{corollary}[theorem]{Corollary}

\newtheorem{lemma}[theorem]{Lemma}

\newtheorem{proposition}[theorem]{Proposition}

\theoremstyle{remark}
\newtheorem{remark}[theorem]{Remark}
\theoremstyle{remark}
\newtheorem{example}[theorem]{Example}
\theoremstyle{remark}

\newtheorem{foo}[theorem]{Remarks}

\def\me{{\mathbb  E}}

\def\mr{{\mathbb  R}}

\numberwithin{equation}{section}

\title{Local times of stochastic differential equations driven by fractional Brownian motions}

\author{Shuwen Lou\thanks{Dept. Mathematics, Statistics and Computer Science, University of Illinois at Chicago, Chicago, IL 60607. Email: slou@uic.edu} }

\author{Cheng Ouyang \thanks{Dept. Mathematics, Statistics and Computer Science, University of Illinois at Chicago, Chicago, IL 60607. Email: couyang@math.uic.edu. CO's research is supported in part by  Simons grant \#355480.}}
\affil{University of Illinois at Chicago}

\begin{document}

\maketitle


\begin{abstract}
In this paper, we study the existence and (H\"{o}lder) regularity of local times of stochastic differential equations driven by fractional Brownian motions. In particular, we show that in one dimension and in the rough case $H<1/2$, the H\"{o}lder exponent (in $t$) of the local time is $1-H$, where $H$ is the Hurst parameter of the driving fractional Brownian motion.
\end{abstract}



\section{Introduction}
In this paper, we consider the following stochastic differential equation (SDE)
\begin{align}\label{sde-intro}
X_t=x+\int_0^t V_0(X_s)ds+\sum_{i=1}^d \int_0^t V_i(X_s)dB^i_s,\quad t\in[0,T],
\end{align}
where $x \in \mathbb{R}^d$, $V_0,V_1,\cdots,V_d$ are $C^\infty$-bounded  vector fields on $\mr^d$ and $\{B_t\}_{0\leq t\leq T}$ is a $d$-dimensional  fractional Brownian motion with Hurst parameter $H\in(1/4,1)$. Throughout our discussion, we assume that the vector fields $V_i$ satisfy the uniform elliptic condition. When $H\in(1/2,1)$, the above equation is understood in Young's sense. When $H\in(1/4,1/2)$,  stochastic integrals in equation (\ref{sde-intro}) are  interpreted as rough path integrals (see, e.g., ~\cite{FV-bk, Gu}) which extends the Young's integral. Existence and uniqueness of solutions to the above equation can be found, for example, in \cite{LQ}. In particular, when $H=\frac{1}{2}$, this notion of solution coincides with the solution of the corresponding Stratonovitch stochastic differential equation.  It is also clear now (cf. \cite{BH,CF,CLL,H-P}) that under H\"{o}rmander's condition the law of the solution $X_t$ has a smooth density with respect to 
the Lebesgue measure on $\mathbb{R}^d$. 

We are interested in the existence and regularity of local times of the solution $X$ to equation (\ref{sde-intro}). For a  $d$-dimensional fractional Brownian motion $B$ itself, its local time has been studied intensively under the framework of Gaussian random fields and is now well-understood  (see, e.g., \cite{Berman2}, \cite{GH} and \cite{BYOZ-book}).    The challenge to investigae local times of  $X$  is that it is not a Gaussian process in general. Many tools developed for  Gaussian random fields can not be directly applied. For example, the Fourier transform of the law of $X$ is not easy to analyze in such a non-Gaussian setting.   

Our approach relies on a sharp estimate of the joint density of finite distributions of $X$ (Theorem \ref{multi point density} below). Essentially, this density estimate plays a similar role to the ``local nondeterminism'' condition that is often used in the content of Gaussian random fields.  The main result of our investigation is summarized as follows.

Fix any small positive number $a$. Let $L(t,x)$ be the local time (occupation density) of $X$ up to time $t$ and $L^a(t,x)$ the occupation density of $X$ over the time interval $[a,t]$.  Define the pathtwise H\"{o}lder exponent of $L^a(\cdot, x)$  by
\begin{equation}\label{def-pointwise-holder}
\alpha(t)=\sup \left\{ \alpha>0, \limsup_{\delta\rightarrow 0}\sup_{x\in \IR^d}\frac{L^a(t+\delta, x)-L^a(t, x)}{\delta^\alpha}=0 \right\}.
\end{equation}

\begin{theorem} Let $X$ be the solution to equation (\ref{sde-intro}).

\begin{itemize}
\item [(1)]\ \  When $dH<1$,  the local time $L(t,x)$ of $X$ exists almost surely for any fixed $t$.
\item [(2)]\ \  Assume $d=1$ and $1/4<H<1/2$. There exists a version of $L^a(t,x)$ that is jointly continuous in $(t,x)$. Moreover, for any $\beta<1-H$, $L^a(t,x)$ is $\beta$-H\"{o}lder continuous in $t$, uniformly in $x$. And its pathwise H\"{o}lder exponent is given by
\begin{equation*}
\alpha(t)=1-H, \quad \mathrm{a.s.}\quad  \text{for all }t\in[a,T].
\end{equation*}
\end{itemize}
\end{theorem}
Finally, let us briefly explain why we have  to impose the technical assumption $d=1$ and $1/2<H<1/4$ for the H\"{o}lder regularity of $L^a(t,x)$. In order to establish the continuity of $L^a(t,x)$ in the space variable, the natural approach is to provide an upper bound for
$$\IE|L^a(t,x)-L^a(t,y)|^n.$$
Our observation is that the above quantity can be bounded from above by
$$|x-y|^n\int_{[a,T]^n} |\partial^n_xp_{t_1,...,t_n}(x_1,...,x_n)|dx_1,...,dx_n,$$
where $p_{t_1,...,t_n}(x_1,...,x_n)$ is the joint density of $(X_{t_1},...,X_{t_n})$. 
In general, one needs $n>d$ in order to conclude the continuity of $L^a(t,x)$ in $x$. However, as one can see from Theorem \ref{multi point density}, large $n$ and $H$ tend to blow up the time integral above.  Consequently, we have to restrict our discussion to the assumption $n=2$ and $(1+d)H<1$, i.e., $d=1$ and $H<1/2$.

The rest of the paper is organized as follows. In Section 2, we introduce some basic tools for analyzing SDEs driven by fractional Brownian motions. In particular, we establish the key estimate for the joint density of $(X_{t_1},...,X_{t_n})$, which enables us to establish both the existence of local time in Section 2 and the regularity of the local time in Section 3.

\section{SDEs driven by fractional Brownian motions}
In this section, we present some tools for analyzing SDEs driven by fractional Brownian motions which will be needed for the remainder of the paper. 

Let $B=\{B_t=(B^1_t,\ldots,B^d_t), \; t\in [0,T]\}$ be a $d$-dimensional fractional Brownian motion with Hurst
parameter $H\in(0,1)$. That is,  $B$ is a centered Gaussian process whose covariance structure is induced by
\begin{align}\label{covariance}
R\left( t,s\right) :=\EE  B_s^i \, B_t^j
=\frac{1}{2}\left( s^{2H}+t^{2H}-|t-s|^{2H}\right)\delta_{ij},
\quad
i, j=1,\ldots,d.
\end{align}
It can be shown, by a standard application of Kolmogorov's criterion, that $B$ admits a continuous version
whose paths are $\ga$-H\"older continuous for any $\ga<H$.

\smallskip

\subsection{Malliavin calculus} We introduce the basic framework of Malliavin calculus in this subsection.  The reader is invited to read the corresponding chapters in  \cite{Nu06} for further details. Let $\mathcal{E}$ be the space of $\mathbb{R}^d$-valued step
functions on $[0,1]$, and $\mathcal{H}$  the closure of
$\mathcal{E}$ for the scalar product:
\[
\langle (\mathbf{1}_{[0,t_1]} , \cdots ,
\mathbf{1}_{[0,t_d]}),(\mathbf{1}_{[0,s_1]} , \cdots ,
\mathbf{1}_{[0,s_d]}) \rangle_{\mathcal{H}}=\sum_{i=1}^d
R(t_i,s_i).
\]
 $\ch$ is the reproducing kernel Hilbert space for $B$. 

Some isometry arguments allow to define the Wiener integral $B(h)=\int_0^{1} \langle h_s, dB_s \rangle$ for any element $h\in\ch$, with the additional property $\EE[B(h_1)B(h_2)]=\langle h_1,\, h_2\rangle_{\ch}$ for any $h_1,h_2\in\ch$.
  A $\mathcal{F}$-measurable real
valued random variable $F$ is said to be cylindrical if it can be
written, for a given $n\ge 1$, as
\begin{equation*}
F=f\lp  B(\phi^1),\ldots,B(\phi^n)\rp,
\end{equation*}
where $\phi^i \in \mathcal{H}$ and $f:\mathbb{R}^n \rightarrow
\mathbb{R}$ is a $C^{\infty}$ bounded function with bounded derivatives. The set of
cylindrical random variables is denoted by $\mathcal{S}$.

The Malliavin derivative is defined as follows: for $F \in \mathcal{S}$, the derivative of $F$ is the $\mathbb{R}^d$ valued
stochastic process $(\mathbf{D}_t F )_{0 \leq t \leq 1}$ given by
\[
\mathbf{D}_t F=\sum_{i=1}^{n} \phi^i (t) \frac{\partial f}{\partial
x_i} \left( B(\phi^1),\ldots,B(\phi^n)  \right).
\]
More generally, we can introduce iterated derivatives by
$
\mathbf{D}^k_{t_1,\ldots,t_k} F = \mathbf{D}_{t_1}
\ldots\mathbf{D}_{t_k} F.
$
\ For any $p \geq 1$,  we denote by
$\mathbb{D}^{k,p}$ the closure of the class of
cylindrical random variables with respect to the norm
\[
\left\| F\right\| _{k,p}=\left( \mathbb{E}\left( F^{p}\right)
+\sum_{j=1}^k \mathbb{E}\left( \left\| \mathbf{D}^j F\right\|
_{\mathcal{H}^{\otimes j}}^{p}\right) \right) ^{\frac{1}{p}},
\]
and
\[
\mathbb{D}^{\infty}=\bigcap_{p \geq 1} \bigcap_{k
\geq 1} \mathbb{D}^{k,p}.
\]

Let $F=(F^1,\ldots , F^n)$ be a random vector whose components are in $\mathbb{D}^\infty$. Define the Malliavin matrix of $F$ by
$$\gamma_F=(\langle \mathbf{D}F^i, \mathbf{D}F^j\rangle_{\ch})_{1\leq i,j\leq n}.$$
Then $F$ is called  {\it non-degenerate} if $\gamma_F$ is invertible $a.s.$ and
$$(\det \gamma_F)^{-1}\in \cap_{p\geq1}L^p(\Omega).$$
\noindent
It is a classical result that the law of a non-degenerate random vector admits a smooth density with respect to the Lebesgue measure on $\mr^n$. 

It is well-known that for a fractional Brownian motion $B$, there is an underlying Wiener process $W$  such that
\begin{align*}
B_t=\int_0^tK_H(t,s)dW_s,
\end{align*}
where $K(t,s)$ is a deterministic kernel whose expression is explicit.  Based on the above representation, one can consider fractional Brownian motions and hence functionals of fractional Brownian motions as functionals of the underlying Wiener process $W$. This observation allows us to perform Malliavin calculus with respect to the Wiener process $W$. We shall perform Malliavin calculus with respect to both $B$ and $W$. In order to distinguish them, the Malliavin derivatives (and corresponding Sobolev spaces, respectively)  with respect to $W$ will be denoted by $D$  (and by $D^{k,p}$, respectively).
The relation between the two operators $\mathbf{D}$ and $D$  is given by the following (see e.g. \cite[Proposition 5.2.1]{Nu06}).

\begin{proposition} 
Let ${D}^{1,2}$ be the Malliavin-Sobolev space corresponding to
the Wiener process $W$. Then $\D^{1,2}={D}^{1,2}$, and for
any $F\in {D}^{1,2}$ we have ${\mathrm D} F = K^{*} \bd F$
,where $K^*$ is the isometry between $L^2([0,T])$ and $\mathcal{H}$. 
\end{proposition}




It is known that $B$ and the Wiener process $W$ generate the same filtration which we denote by $\{\mathcal{F}_t; \, t\in[0,1]\}$. Set $L^{2}_{t}\equiv L^{2}([t,1])$ and $\mathbb{E}_t=\mathbb{E}( \cdot\, | \mathcal{F}_t)$. For a random variable $F$ and $t\in [0,1]$, define,  for $k\ge 0$ and $p>0$, the conditional Sobolev norm
\begin{equation*}
\|F\|_{k,p;t}=\left( \mathbb{E}_t\left[
F^{p}\right]
+\sum_{j=1}^k \mathbb{E}_t\left[ \left\| {D}^j F\right\|
_{(L^{2}_{t})^{\otimes j}}^{p}\right] \right) ^{\frac{1}{p}}.
\end{equation*}
By convention, when $k=0$ we always write $\|F\|_{p;t}=\|F\|_{0,p;t}$.
The conditional Malliavin matrix of $F$ is given by
\begin{align}\label{eq:def-conditional-matrix}
{\Gamma}_{F,t}=
\lp\langle D F^i, D F^j\rangle_{L^{2}_{t}}\rp_{1\leq i,\,j\leq n}.
\end{align}

\smallskip
The following is a conditional version of Proposition 2.1.4 of \cite{Nu06}.
\begin{proposition} \label{prop:int-parts-cond-W}
Fix $k\geq 1$. Let $F=(F^1,...,F^n)$ be a random vector and $G$ a random variable. Assume both $F$ and $G$ are smooth in the Malliavin sense and
$(\det{{\Gamma}_{F,s}})^{-1}$ has finite moments of
all orders.  Then for any
multi-index
$\alpha\in\{1,\,\ldots,\,n\}^k, k\geq1$,
there exists an element ${H}^s_\alpha(F,G)\in\cap_{p\geq
1}\cap_{m\geq 0} {D}^{m,p}$ such that
\begin{equation*}
\EE_s\lc(\partial_\alpha \varphi)(F) \, G\rc
=
\EE_s\lc \varphi(F) \, {H}_\alpha^s(F,G)\rc,\quad\ \   \varphi\in\cac_p^\infty(\mr^d),
\end{equation*}
where ${H}_\alpha^s(F,G)$ is
recursively defined by
\begin{equation*}
{H}_{(i)}^s(F,G)=\sum_{j=1}^n {\delta}_s\lp
G\lp{{\Gamma}}_{F,s}^{-1}\rp_{ij}D F^j\rp,
\quad
{H}_{\alpha}^s(F,G)={H}^s_{(\alpha_k)}(F,{H}^s_{(\alpha_1,\,\ldots,\,\alpha_{k-1})}(F,G)).
\end{equation*}
Here ${\delta}_s$ denotes the Skorohod integral with respect
to the Wiener process $W$ on the interval $[s,1]$. (See \cite[Section 1.3.2]{Nu06} for a detailed account of the definition of $\delta_s$.) 
Furthermore, the following norm estimate holds true:
\begin{equation*}
\Vert  H_{\alpha}^{s}(F,G) \Vert_{p;s} \leq 
c_{p,q} \Vert \Gamma_{F,s}^{-1}  \, D F\Vert^k_{k, 2^{k-1}r;s}\Vert G\Vert^k_{k, q;s},
\end{equation*}
 where $\frac1p=\frac1q+\frac1r$.
\end{proposition}

\bigskip

\subsection{SDEs driven by fractional Brownian motions and density estimate}
Consider the following stochastic differential equation driven by a fractional Brownian motion with Hurst parameter $H>1/4$,
\begin{align}\label{SDE}
X_t=x_0+\sum_{i=1}^d\int_0^tV_i(X_s)dB^i_s+\int_0^tV_0(X_s)ds,\quad\quad t\in[0,T].
\end{align}
Here $V_i, i=0,1,...,d$ are $C^\infty$-bounded vector fields on $\mr^d$ which form a uniform elliptic system.  


Recall that $D$ is the Malliavin derivative operator with respect to the underlying Wiener process $W$ and $\Gamma_{F,t}$ is defined in (\ref{eq:def-conditional-matrix}) for a random variable $F$. The following estimate is a restatement of Proposition 5.9 in  \cite{BNOT}.

\begin{lemma} \label{estimate bivar mall} Let $a \in (0,T)$, and consider $H\in(1/4,1)$. Then there exists a constant $C>0$ depending on $a$ such that for $ a \le s \le t \le T$ the following holds:
 \begin{eqnarray*}
 \Vert \Gamma_{X_{t}-X_{s},s}^{-1}   \Vert^d_{d, 2^{d+2};s} 
 &\le& 
 \frac{C}{(t-s)^{2dH}} \, \big(\mathbb{E}_s (1+ G)\big)^{\frac{d}{2^{d+2}}}  \\
 \Vert D(X_{t}-X_{s}) \Vert^d_{d, 2^{d+2};s} 
 &\le& 
 C (t-s)^{dH} \, \big(\mathbb{E}_s  (1+  G)\big)^{\frac{d}{2^{d+2}}},
\end{eqnarray*}
where $G$ is a random variable smooth in the Malliavin sense and has finite moments to any order.
\end{lemma}

The above lemma allows us to estimate the joint density of $(X_{t_1},...,X_{t_n})$, or equivalently the joint density of $(X_{t_1}, X_{t_2}-X_{t_1},..., X_{t_n}-X_{t_{n-1}})$.

\begin{theorem}\label{multi point density}
Fix $a\in(0,T)$ and $\gamma<H$. Let $\tilde{p}_{t_1,...,t_n}(\xi_1,...,\xi_n)$ be the joint density of the random vector $(X_{t_1}, X_{t_2}-X_{t_1},..., X_{t_n}-X_{t_{n-1}}),$   $a\leq t_1<\cdots<t_n\leq T$. For any non-negative integers $k_1,...,k_n$, there exists a constant $C_1, C_2>0$ (depending on $a$) such that 
\begin{align*}
&\partial^{k_1}_{\xi_1}\dots\partial^{k_n}_{\xi_n}\tilde{p}_{t_1,...,t_n}(\xi_1,...,\xi_n)\\
&\leq C_1\,\frac{1}{(t_2-t_1)^{(d+k_2)H}} e^{-\frac{|\xi_2|^{2\gamma}}{2|t_2-t_1|^{2\gamma^2}}}\dots \frac{1}{(t_{n}-t_{n-1})^{(d+k_n)H}} e^{-\frac{|\xi_n|^{2\gamma}}{C_2|t_n-t_{n-1}|^{2\gamma^2}} }.
\end{align*}
\end{theorem}
\begin{remark}
In the above upper bound, the term $t_1^{-(d+k_1)H}$ has been absorbed in the constant $C_1$, given the fact that $t_1\geq a>0$.
\end{remark}
\begin{proof} The proof is similar to the one for Theorem 3.5 in \cite{LO}.  We only prove the case of $n=3$. The general case is almost identical. The positive constants $c_i$, $1\le i\le 4$, may change from line to line.

First observe that  $\tilde{p}_{t_1, t_2, t_3}$ can be expressed as
\begin{align*}
&\partial^{k_1}_{\xi_1}\partial^{k_2}_{\xi_2}\partial^{k_3}_{\xi_3}\tilde{p}_{t_1, t_2, t_3}(\xi_{1},\xi_{2},\xi_3)\\
= &\partial^{k_1}_{\xi_1}\partial^{k_2}_{\xi_2}\partial^{k_3}_{\xi_3}\EE\lc \delta_{\xi_1}(X_{t_1}) \, \delta_{\xi_2}(X_{t_2}-X_{t_1}) \, \delta_{\xi_3}(X_{t_3-t_2}) \,  \rc,
\quad\text{for}\quad
\xi_{1},\xi_{2},\xi_3\in\R^{d},\\
 = &\partial^{k_1}_{\xi_1}\partial^{k_2}_{\xi_2}\EE\lc \delta_{\xi_1}(X_{t_1}) \,\delta_{\xi_2}(X_{t_2}-X_{t_1})\,\partial^{k_3}_{\xi_3} \EE_{t_2}\lc \delta_{\xi_3}(X_{t_3}-X_{t_2})  \rc \rc.
\end{align*}
Due to Proposition \ref{prop:int-parts-cond-W},  $M_{t_2t_3}=\partial^{k_3}_{\xi_3} \EE_{t_2}\lc \delta_{\xi_2}(X_{t_3}-X_{t_2})  \rc $ can be bounded from above as follows:
\begin{align*}
|M_{t_2t_3}| 
\leq  c_1 \Vert \Gamma_{X_{t_3}-X_{t_2},t_2}^{-1}   \Vert^{d+k_3}_{d+k_3, 2^{d+k_3+2},t_2} \, \Vert D(X_{t_3}-X_{t_2}) \Vert^{d+k_3}_{d+k_3, 2^{d+k_3+2};t_2}  \big(\EE_{t_2}[ \1_{(X_{t_3}-X_{t_2} > \xi_{3})}]\big)^{1/2},
\end{align*}
where $c_1$ is some positive constant. Therefore, Lemma \ref{estimate bivar mall} yields that for some constant $c_2>0$,
\begin{align}\label{first-bnd-bivariate}
&\partial^{k_1}_{\xi_1}\partial^{k_2}_{\xi_2}\partial^{k_3}_{\xi_3}{p}_{t_1, t_2, t_3}(\xi_{1},\xi_{2},\xi_3)\\
 \le & 
\frac{c_2}{(t_3-t_2)^{(d+k_3)H}}   
\partial^{k_1}_{\xi_1}\partial^{k_2}_{\xi_2}\EE\lc \delta_{\xi_1}(X_{t_1}) \, \delta_{\xi_2}(X_{t_2}-X_{t_1})\,
\big(\mathbb{E}_{t_2} (1+ G)^2\big)^{\frac{d+k_3}{2^{d+k_3+1}}} \,
\big(\EE_{t_2}[ \1_{(X_{t_3}-X_{t_2} > \xi_{3})}]\big)^{1/2} \rc \nonumber
\end{align}

To proceed, denote by $\mathbf{B}$ the lift of $B$ as a rough path. Set
\[
\cn_{\gamma,q}(\mathbf{B})=\int_{0}^{T}\int_{0}^{T}
\frac{d(\mathbf{B}_v,\mathbf{B}_u)^{q}}{|v-u|^{\gamma q}}\,dudv,
\]
where $\gamma <H$ and $q$ a nonnegative integer. $\cn_{\gamma,q}(\mathbf{B})$ is the Besov norm of $\mathbf{B}$ (as a function in $t$). The reason we need to consider these Besov norms $\cn_{\gamma,q}(\mathbf{B})$ is that they are  smooth in the Malliavin sense. 

Fix any $\varepsilon>0$ satisfying $\gamma+\varepsilon<H$. It has been shown in \cite[Theorem 3.5]{LO} that for some $c>0$,
\begin{equation*}
|X_t-X_s| \le c\, |t-s|^\gamma 
(1+\cn_{\ga+\epsilon,2q}(\mathbf{B})^{1/2q})^{1/\gamma},
\end{equation*}
and there exists a constant $\lambda_0>0$ such that 
$$\me \exp\left\{\lambda_0 \cn_{\ga+\epsilon,2q}(\mathbf{B})^{1/q} \right\}<\infty.$$
That is, $\cn_{\gamma+\epsilon,2q}(\mathbf{B})^{1/2q}$ has Gaussian tail.
Thus, for $\lambda<\lambda_0$ we can find a constant $c_3>0$,
\begin{equation*}
\EE_{t_2}\left[ \1_{(X_{t_3}-X_{t_2} > \xi_{3})} \right] \le 
 c_3  \exp\left\{-\lambda\frac{|\xi_3|^{2\gamma}}{|t_3-t_2|^{2\gamma^2}} \right\}\EE_{t_2}\exp\left\{ \lambda (1+\cn_{\ga+\epsilon,2q}(\mathbf{B}))^{2/2q}\right\} .
\end{equation*}
Plugging this inequality into \eqref{first-bnd-bivariate}, we end up with:
\begin{align}\label{second-bnd-bivariate}
&\partial^{k_1}_{\xi_1}\partial^{k_2}_{\xi_2}\partial^{k_3}_{\xi_3}\tilde{p}_{t_1, t_2, t_3}(\xi_{1},\xi_{2},\xi_3)\\
\le&  \frac{c_3}{(t_3-t_2)^{(d+k_3)H}} e^{-\lambda\frac{|\xi_3|^{2\gamma}}{c_4|t-s|^{2\gamma^2}} }  \,
\partial^{k_1}_{\xi_1}\partial^{k_2}_{\xi_2}\EE\left[ \delta_{\xi_1}(X_{t_1}) \,\delta_{\xi_2}(X_{t_2}-X_{t_1})\, \Psi  \right]\nonumber
\end{align}
where $\Psi$ is a random variable that is smooth in the Malliavin calculus sense. [$\Psi$ only depends on $G$ and $ \cn_{\ga+\epsilon,2q}(\mathbf{B})^{1/2q}$.]

Now we can continue our computation by writing
$$\partial^{k_1}_{\xi_1}\partial^{k_2}_{\xi_2}\EE\left[ \delta_{\xi_1}(X_{t_1}) \,\delta_{\xi_2}(X_{t_2}-X_{t_1})\, \Psi  \right]=\partial^{k_1}_{\xi_1}\EE\left[ \delta_{\xi_1}(X_{t_1}) \,\partial^{k_2}_{\xi_2}\EE_{t_1}[\delta_{\xi_2}(X_{t_2}-X_{t_1})\, \Psi  ]\right],$$
and repeating the above argument to establish an upper bound for $M_{t_2t_1}=\partial^{k_2}_{\xi_2}\EE_{t_1}[\delta_{\xi_2}(X_{t_2}-X_{t_1})\, \Psi  ]$. Observe that the existence of a smooth random variable $\Psi$ does not affect the upper bound estimate when applying integration by parts. [But we may need to choose $\lambda$ small enough to ensure sufficient integrability for $\Psi$ and its Malliavin derivatives.] Hence, 
\begin{align}\label{third-bnd-bivariate}
&\partial^{k_1}_{\xi_1}\partial^{k_2}_{\xi_2}\partial^{k_3}_{\xi_3}\tilde{p}_{t_1, t_2, t_3}(\xi_{1},\xi_{2},\xi_3)\\
\le&  c_3\frac{1}{(t_3-t_2)^{(d+k_3)H}} e^{-\lambda\frac{|\xi_3|^{2\gamma}}{c_4|t_3-t_2|^{2\gamma^2}} } \, \frac{1}{(t_2-t_1)^{(d+k_2)H}} e^{-\lambda\frac{|\xi_2|^{2\gamma}}{c_4|t_2-t_1|^{2\gamma^2}} }  \,
\partial^{k_1}_{\xi_1}\EE\left[ \delta_{\xi_1}(X_{t_1}) \, \Phi  \right]\nonumber
\end{align}
where $\Phi$ is a random variable that is smooth in the Malliavin calculus sense. [$\Phi$ still only depends on $G$ and $ \cn_{\ga+\epsilon,2q}(\mathbf{B})^{1/2q}$.]
 We can now integrate \eqref{third-bnd-bivariate} safely by parts in order to regularize the term $\delta_{\xi_1}(X_s)$, which finishes the proof.

\end{proof}

The following technical result will be needed later. Let $p_{t_1,...,t_n}(x_1,...,x_n)$ be the joint density of $(X_{t_1},...,X_{t_n}), a\leq t_1\leq,...,\leq t_n\leq T$, and set
\begin{align}\label{v}
v(x_1,...,x_n)=\int_{[a,T]^n}p_{t_1,...,t_n}(x_1,...,x_n)dt_1,...,dt_n.
\end{align}

\begin{corollary}\label{v continuity}
If $(1+d)H<1$, then $v$ is uniformly continuous in $(x_1,...,x_n)$.
\end{corollary}
\begin{proof}
Let $\bar{x}=(x_1,...,x_n)$ and $\bar{y}=(y_1,...,y_n)$.
\begin{align}
|v(\bar{x})-v(\bar{y})|\leq&\int_{[a,T]^n}|p_{t_1,...,t_n}(\bar{x})-p_{t_1,...,t_n}(\bar{y})|dt_1,...,dt_n\nonumber\\
\leq&|\bar{x}-\bar{y}|\int_{[a,T]^n} |\partial_{\bar{x}}p_{t_1,...,t_n}(\bar{x}^*)|dt_1,...,dt_n.\label{vx-vy}
\end{align}
Observe that
$$p_{t_1,...,t_n}(x_1,...,x_n)=\tilde{p}_{t_1,...,t_n}(x_1, x_2-x_1, x_3-x_2,...,x_n-x_{n-1}).$$ 
It follows immediately from Theorem \ref{multi point density} that
\begin{align}\label{bound first derivative p}|\partial_{\bar{x}}p_{t_1,...,t_n}(\bar{x})|\leq&c_1\sum_{k=2}^n\frac{1}{|t_2-t_2|^{dH}}\cdots\frac{1}{|t_k-t_{k-1}|^{(1+d)H}}\cdots\frac{1}{|t_n-t_{n-1}|^{dH}}\nonumber\\
\leq& c_2\frac{1}{|t_2-t_1|^{(1+d)H}}\cdots\frac{1}{|t_n-t_{n-1}|^{(1+d)H}}.\end{align}
By combining (\ref{vx-vy}) and (\ref{bound first derivative p}), and since there is no singularity in the integral (\ref{vx-vy}) in $t$ under our assumption $(1+d)H<1$, the proof is complete.
\end{proof}

\section{Existence of local time}
For any time interval $I\subset[0,T]$, define $$L(I, A):=\lambda(\{s\in I: X_s\in A\}),\quad\quad A\subset\IR^d,$$ where $\lambda$ is the Lebesgue measure on $\mr^+$. It is clear that $L(I,\cdot)$ is the occupation measure on $\IR^d$ of $X$ during time interval $I$. If $L(I,\cdot)$ is absolutely continuous with respect to the Lebesgue measure $\lambda_d$ on $\IR^d$, we denote the Radon-Nikodym derivative  by $$L(I,x)=\frac{dL(I,\cdot)}{d\lambda_d},$$ and call $L(I, x)$ the occupation density of $X$ at $x$ during time interval $I$. Thus it holds for  all measurable $A\subset \IR^d$ that
\begin{equation*}
L(I, A)=\int_{x\in A}L(I,x)dx.
\end{equation*}
Whenever there is no confusion, we always write $L(t,A)=L([0,t],A)$ and $L(t,x)=L([0,t],x)$ for the occupation measure and occupation density, respectively.

The following theorem is a classical result  on the existence of occupation density (local time) for a stochastic process $X$ (\cite[Theorem 21.12]{GH}).
\begin{theorem}\label{existence criteria}
For any interval $I\subset\mr^+$, suppose

\begin{equation*}
\liminf_{\epsilon\downarrow 0}\frac{1}{\epsilon^d}\int_I \IP\{\left(\left|X_s-X_u\right|\right)\leq\epsilon\}ds<\infty\quad \text{for\ a.e.}\ u\in I,
\end{equation*}
then the occupation density $L(I,x)$ exists almost surely.
\\
\end{theorem}
On account of the above theorem, we are ready to state our result on the existence of local time for the solution to equation (\ref{SDE}).
\begin{theorem}
Let $X$ be the solution to equation (\ref{SDE}). Assume that $dH<1$, then for each fixed $t\in[0,T]$,  $X$ has local time $L(t,x)$ almost surely.
\end{theorem}
\begin{proof}
Fix a  $t\in[0,T]$ and $a\in(0,t]$. It follows from Theorem \ref{multi point density} that the probability density  $\tilde{p}_{s,u}(z)$ of $X_u-X_s$ satisfies that for some $c_1, c_2>0$, 
\begin{equation*} 
\tilde{p}_{s,u}(z) \leq  \frac{c_1}{(u-s)^{dH}} \exp\left\{-\frac{|z|^{2\gamma}}{c_2(u-s)^{2\gamma^2}} \right\},\; \text{ for all } a\leq s<u\leq T,
\end{equation*}
where $\gamma<H$. It follows immediately that under our assumption $dH<1$, we have for the solution $X$ to equation (\ref{SDE})
\begin{equation*}
\liminf_{\epsilon\downarrow 0}\frac{1}{\epsilon^d}\int_a^t \IP\{\left(\left|X_s-X_u\right|\right)\leq\epsilon\}ds<\infty\quad \text{for\ all}\ u\in [a,t].
\end{equation*}
Thus, by Theorem \ref{existence criteria}, there exists a null set $N_a\subset\Omega$ such that  the occupation density $L([a,t],x)(\omega)$ exists for all $\omega\notin N_a$. Since the occupation density exists for all $\omega\notin N_a$, we can work on the following modification (in $x$) of it (for all $\omega\notin N_a$), which is still denoted by $L([a,t],x)$,
$$L([a,t],x)=\liminf_{\epsilon\downarrow0}\frac{L([a,t],B(x,\epsilon))}{C_d\epsilon^d}=\liminf_{\epsilon\downarrow0}\frac{1}{C_d\epsilon^d}\int_a^t\mathbf{1}_{B(x,\epsilon)}(X_s)ds.$$
Here, $C_d$ is such that $C_d\epsilon^d$ gives the volume of the ball $B(x,\epsilon)$ in $\IR^d$. The reason to work on this version of $L([a,t],x)$ is that it is monotone in $a$. 

Observe that we have for all $\omega\notin N_a$,
$$L([a,t], A)=\int_{x\in A}L([a,t],x)dx,\quad\quad\text{for\ all\ Broel}\ A\subset\IR^d.$$
Letting $a\downarrow0$ (along a sequence $a_n$) in the above and by monotone convergence theorem, we have for all $\omega\notin \cup_{n=1}^\infty N_{a_n}$
$$L([0,t],A)=\lim_{a_n\downarrow0}L([a_n,t], A)=\int_{x\in A}\lim_{a_n\downarrow0}L([a_n,t],x)dx.$$
Hence, almost surely, the occupation density $L([0,t],x)$ of the occupation measure $L([0,t],A)$ exists (which is $\lim_{a_n\downarrow0}L([a_n,t],x)$).
\end{proof}

\section{Regularity of local time}
Unless otherwise stated, it is assumed that $(1+d)H<1$ throughout this section. As a remark, in the case $d=1$, any $H<1/2$ satisfies this condition.

Fix a small positive number $a\in(0,T]$, and let $L^a(t,A)$ and $L^a(t,x)$ be the occupation measure and occupation density respectively on the time interval $[a,t]$. That is
$$L^a(t,A)=L([a,t], A),\quad \text{and}\ \ L^a(t,x)=L([a,t],x).$$ 
In order to use Kolmogorov type continuity theorems to conclude various continuities of $L^a(t,x)$ in $t$ and $x$,  we seek to provide upper bounds for quantities in the form of $\IE|L^a(t,x)-L^a(t,y)|^n$. For this purpose, it is necessary to specify, for each fixed $t$ and $x$, a version of $L^a(t,x)$ more carefully. Let 
\begin{align}\label{approximation L}L_0^a(t,x)=\liminf_{\epsilon\downarrow0}\frac{1}{C_d\epsilon^d}\int_a^t\mathbf{1}_{B(x,\epsilon)}(X_s)ds,\end{align}
which, by our discussion before, exists as a finite limit for a.e. $x$ almost surely. 
Set
\begin{equation}
L_\epsilon^a(t,x)=\frac{1}{C_d\epsilon^d}\int_a^t\mathbf{1}_{B(x,\epsilon)}(X_s)ds
\end{equation} Recall the content of Corollary \ref{v continuity}, standard argument shows that the uniform continuity of $v$ implies that $L_\epsilon^a(t,x)$ is Cauchy in $L^n(\IP)$, uniformly in $x$. [See the argument in p42-p43 in \cite{GH}.] It follows that there is a subsequence $\{\epsilon_n\}_{n\ge 1}$ such that $L_{\epsilon_n}^a(t,x)$ converges uniformly in $L^n(\IP)$ and converges almost surely for each fixed $x$. Hence the limit exists in (\ref{approximation L}) through the subsequence $\epsilon_n$ both uniformly in $L^n(\IP)$ and almost surely for each fixed $x$, and there is no problem then in arguing as if $\epsilon_n$ is the full sequence $\epsilon$ in the limit in (\ref{approximation L}). This limit, which we denote by $L^a(t,x)$, is the version we work with throughout our discussion below. 

The following theorem is a restatement of Theorem 3.1 of \cite{Berman}.
\begin{theorem}\label{Berman}
Let  $Y(t, x), (t,x)\in [0,T]\times I_R$ be a stochastic process of two parameters. Suppose there are positive constants $n, C$ and $c,d$ such that
\begin{equation}\label{127303}
\IE|Y(t, x+k)-Y(t, x)|^n\le C|k|^{1+c}, \quad \text{for }x, x+k\in I_R, t=0,T;
\end{equation}
\begin{equation}\label{127305}
\IE|Y(t+h,x)-Y(t,x)|^n\le C|h|^{1+d},\quad \text{for all }t,t+h\in [0,T], \ \mathrm{and}\  x \in I_R;
\end{equation}
\begin{align}\label{127306}
\IE|Y(t+h, x+k)-Y(t, x+k)-&Y(t+h, x)+Y(t,x)|^n\le C|k|^{1+c}|h|^{1+d},\nonumber\\
& \text{ for all } t,t+h\in [0,T]\ \mathrm{and}\ x, x+k \in I_R.
\end{align}
Then there exists a version of the process $Y$ jointly continuous in $(t,x)$. Moreover, for every $\gamma<d/r$ there exist random variables $\eta$ and $\Delta$ which are almost surely positive and finite such that
\begin{align*}
|Y(t+h,x)-Y(t,x)|\le \Delta|h|^{\gamma}, \text{ for all } t,t+h\in [0,T],x\in I_R, \text{ and }|h|<\eta.
\end{align*}
\end{theorem}
Our primary goal for this section is to find a version of the local time $L^a(t,x)$ with jointly H\"{o}lder continuity in $(t,x)$. Furthermore, on account of Theorem \ref{Berman}, we show that the H\"{o}lder regularity of the local time in the time variable is uniform in the space variable. The next three lemmas establish the three inequalities in the condition of  Theoerm \ref{Berman}. We start off with the H\"{o}lder continuity in the space variable. 
\begin{lemma}\label{proving-holder-space}
Assume $d=1$ and $1/4<H<1/2$. For any $t\in[a,T]$, and  $\alpha$ satisfying $1<\alpha<\frac{1}{H}-1$, there exists a constant $C_1>0$ (depending on $a$ and $T$) such that for any $x, y\in \IR$, 
\begin{equation}
\IE\left|L^a(t,x)-L^a(t, y)\right|^2\le C_1|x-y|^\alpha. 
\end{equation}
\end{lemma}

\begin{proof}
By our choice of $L^a(t,x)$ in the paragraph after (\ref{approximation L}), and  by Fatou's lemma, 
\begin{align*}\IE\left|L^a(t,x)-L^a(t,y)\right|^2&=\IE\ \lim_{\epsilon\downarrow0}\left|\frac{L^a(t,B(x, \epsilon))-L^a(t, B(y, \epsilon))}{C_d\epsilon}\right|^2\\
&\leq\liminf_{\epsilon\downarrow0} \frac{1}{C_d^2\epsilon^2}\IE\left|L^a(t,B(x, \epsilon))-L^a(t, B(y, \epsilon))\right|^2.\end{align*}

Let $p_{u,s}(z_1,z_2)$ be the joint density of $(X_u,X_s)$. By change of variables, we have
\begin{align*}
&\IE\left|L^a(t,B(x, \epsilon))-L^a(t, B(y, \epsilon))\right|^2
\\
=&2\IE \int_{s=a}^t\int_{u=a}^s \left(\mathbf{1}_{B(x, \epsilon)}(X_u)-\mathbf{1}_{B(y, \epsilon)}(X_u)\right)\left(\mathbf{1}_{B(x, \epsilon)}(X_s)-\mathbf{1}_{B(y, \epsilon)}(X_s)\right)dsdu 
\\
=&2\int_{s=a}^t\int_{u=a}^s \int_{\IR^2}\left(\mathbf{1}_{B(x, \epsilon)}(z_1)-\mathbf{1}_{B(y, \epsilon)}(z_1)\right)\left(\mathbf{1}_{B(x, \epsilon)}(z_2)-\mathbf{1}_{B(y, \epsilon)}(z_2)\right)p_{u,s}(z_1, z_2)dz_1dz_2dsdu 
\\
=&2\int^t_{s=a}\int^s_{u=a}\int \int_{z_1, z_2\in B(x, \epsilon)} \big[ p(z_1, z_2)-p(z_1-(x-y), z_2)
\\
&\quad\quad\quad\quad\quad\quad\quad\quad+p(z_1-(x-y), z_2-(x-y))-p(z_1, z_2-(x-y))\big]dz_2dz_1duds.
\end{align*}
For notation convenience, we introduce 
\begin{align*}
&F_{x,y}(z_1,z_2)\\
&\ \ =p(z_1, z_2)-p(z_1-(x-y), z_2)+p(z_1-(x-y), z_2-(x-y))-p(z_1, z_2-(x-y)).
\end{align*}
By Theorem \ref{multi point density} and Taylor expansion of $F_{x,y}(z_1, z_2)$ to the first and second order (with respect to $x-y$),  there exist a constant $c_1$ such that
\begin{equation}\label{stuff229}
F_{x,y}(z_1,z_2) \le |x-y|\frac{c_1}{|s-u|^{2H}},
\end{equation}
and 
\begin{equation}\label{stuff230}
F_{x,y}(z_1,z_2) \le |x-y|^2\frac{c_1}{|s-u|^{3H}}.
\end{equation}
Write
\begin{align}
&\IE\left|L^a(t,B(x, \epsilon))-L^a(t, B(y, \epsilon))\right|^2=2\int^t_{s=a}\int^s_{u=a}\int \int_{z_1, z_2\in B(x, \epsilon)} F_{x,y}(z_1, z_2)dz_2dz_1duds \nonumber
\\
=&2\int\int_{|s-u|^H>|x-y|}\int \int_{z_1, z_2\in B(x, \epsilon)} F_{x,y}(z_1, z_2)dz_2dz_1duds \nonumber
\\
&\quad+2\int\int_{|s-u|^H\le |x-y|}\int \int_{z_1, z_2\in B(x, \epsilon)} F_{x,y}(z_1, z_2)dz_2dz_1duds. \label{stuff227}
\end{align}
We take care of the two terms on the right hand side of \eqref{stuff227} separately. Let $\delta>0$ be a fixed constant satisfying $d+1+\delta<1/H$. For the first term, it holds
\begin{align}
&\int\int_{|s-u|^H>|x-y|}\int \int_{z_1, z_2\in B(x, \epsilon)} F_{x,y}(z_1, z_2)dz_2dz_1duds \nonumber
\\
\leq&\int\int_{|s-u|^H>|x-y|}\int \int_{z_1, z_2\in B(x, \epsilon)} |x-y|^2\frac{c_2}{|s-u|^{3H}} \;dz_1dz_2duds \quad\quad \quad\quad[\,\text{by}\ (\ref{stuff230})\,] \nonumber
\\
\leq &\ \ c_2\,\epsilon^{2}|x-y|^{1+\delta}\int\int_{|s-u|^H>|x-y|}\frac{1}{|s-u|^{(2+\delta)H}}\;duds \le c_2\,\epsilon^{2}|x-y|^{1+\delta}
\label{stuff238}
\end{align}
Now for the second term on the right hand side of \eqref{stuff227}, due to \eqref{stuff229} it holds
\begin{align}
&\quad \int\int_{|s-u|^H\le |x-y|}\int \int_{z_1, z_2\in B(x, \epsilon)} F_{x,y}(z_1, z_2)dz_2dz_1duds \nonumber
\\
\leq& \int\int_{|s-u|^H\le |x-y|}\int \int_{z_1, z_2\in B(x, \epsilon)} |x-y|\frac{c_3}{|s-u|^{(d+1)H}}\;dz_1dz_2duds \quad\quad\quad[\,\text{by}\ (\ref{stuff229})\,]\nonumber
\\
\le&\  c_3\,\epsilon^{2}|x-y|\int\int_{|s-u|^H\le |x-y|} \frac{1}{|s-u|^{2H}}\;duds \nonumber
\\
\le&\  c_3\, \epsilon^{2}|x-y|\cdot |x-y|^{(1-2H)/H}=c_3\,\epsilon^{2}|x-y|^{\frac{1}{H}-1},
\end{align}
where $\frac{1}{H}-1>1+\delta$. Therefore, the proof is complete by choosing $\alpha=1+\delta$ in the statement of the theorem. 
\end{proof}
It follows immediately by Kolmogorov continuity theorem that, when $d=1$ and $H<1/2$, there exists a modification of $L^a(t, x)$ that is $\alpha$-H\"{o}lder continuous in $x$ for any $\alpha<\frac{1}{2H}-1$.

\bigskip

In the following discussion, we seek to establish the H\"{o}lder regularity of $L^a(t,x)$ in the time variable. We start with two technical lemmas which correspond to \eqref{127305} and \eqref{127306} in Theorem \ref{Berman}
\begin{lemma}\label{126456}
Assume $dH<1$. For all $x\in \IR^d$, and even $n\in \mathbb{N}$, there exists some positive constant $C_2$ (depending on $a$ and $n$) such that for all $x\in \IR^d$, all $a<s<t\le T$,
\begin{equation}
\IE\left|L^a(t,x)-L^a(s, x)\right|^{n}\le C_2|t-s|^{(1-dH)(n-1)+1}. 
\end{equation}
\end{lemma}
\begin{proof}
As before, we only need to find an upper bound for the following:
\begin{align*}
&\qquad \IE\left|L^a(t,B(x, \epsilon))-L^a(s, B(x, \epsilon))\right|^{n} 
\\
&=n! \IE  \int_{u_n=s}^t \int_{u_{n-1}=s}^{u_n}\cdots \int_{u_1=s}^{u_2}\mathbf{1}_{B(x, \epsilon)}(X_{u_1})\cdots \mathbf{1}_{B(x, \epsilon)}(X_{u_n})du_1\cdots du_n 
\\
&=n! \IE  \int_{u_n=s}^t \int_{u_{n-1}=s}^{u_n}\cdots \int_{u_1=s}^{u_2} \int_{z_1, \cdots z_n\in B(x, \epsilon)}p_{u_1, \cdots, u_n}(z_1, \cdots, z_n)dz_1\cdots dz_ndu_1\cdots du_n.
\end{align*}
In the above, $p_{u_1,...,u_n}(z_1,\cdots,z_n)$ is the joint density of $(X_{u_1},...,X_{u_n})$. By Theorem \ref{multi point density}, clearly we have that for some $c_1>0$ (which may change from line to line), 
$$p_{u_1,...,u_n}(z_1,\cdots,z_n)\leq c_1\frac{1}{(u_n-u_{n-1})^{dH}}\cdots \frac{1}{(u_2-u_1)^{dH}},$$
 Hence the right hand side above is upper bounded by
\begin{align}
&c_1n!\epsilon^{nd}\int_{u_n=s}^t \int_{u_{n-1}=s}^{u_n}\cdots \int_{u_1=s}^{u_2}\; \frac{1}{(u_n-u_{n-1})^{dH}}\cdots \frac{1}{(u_2-u_1)^{dH}}\;du_1\cdots du_n \nonumber
\\
\le &c_1n!\epsilon^{nd}\int_{u_n=s}^t \int_{u_{n-1}=s}^{u_n}\cdots \int_{u_2=s}^{u_3}\;\; \frac{1}{(u_n-u_{n-1})^{dH}}\cdots \frac{1}{(u_3-u_2)^{dH}}(u_2-s)^{1-dH}\; du_2\cdots du_n \nonumber
\\
\le &c_1n!\epsilon^{nd}\int_{u_n=s}^t \int_{u_{n-1}=s}^{u_n}\cdots \int_{u_2=s}^{u_3}\;\; \frac{1}{(u_n-u_{n-1})^{dH}}\cdots \frac{1}{(u_3-u_2)^{dH}}(u_3-s)^{1-dH}\; du_2\cdots du_n \nonumber
\\
\le &c_1n!  \epsilon^{nd}\int_{u_n=s}^t \int_{u_{n-1}=s}^{u_n}\cdots \int_{u_2=s}^{u_3}\;\; \frac{1}{(u_n-u_{n-1})^{dH}}\cdots \frac{1}{(u_4-u_3)^{dH}}(u_3-s)^{2(1-dH)}\; du_3\cdots du_n \nonumber
\\
\le &\cdots 
\le c_1n! \epsilon^{nd}\int_{u_n=s}^t (u_n-s)^{(1-dH)(n-1)}du_n 
\le c_1n! \epsilon^{nd}(t-s)^{(1-dH)(n-1)+1} \label{stuff}
\end{align}
The proof is complete.
\end{proof}

Again on account of Kolmogorov continuity theorem, there exists a modification of $L^a(t, x)$ being $\beta$-H\"{o}lder continuous in $t$ for any $\beta<1-dH$. The next lemma is needed for the uniform  H\"{o}lder continuity of the local time in the time variable (uniform in $x$). 
\begin{lemma}\label{proving-uniform-holder}
Assume $d=1$ and $H\in(1/4,1/2)$. For any $\delta\in (0,\frac{1}{H}-2)$ and any even integer $n\in \mathbb{N}$, there exists some $C_3>0$ such that
\begin{equation}\label{126221}
\IE\left[L^a([s,t],x)-L^a([s,t],y)\right]^n\le C_3|x-y|^{1+\delta}\cdot |t-s|^{1+(1-H)(n-3)},\end{equation}
 {for all } $a\le s<t\le T$ and $ x,y\in \IR. $

\end{lemma}
\begin{proof}
Denote by $$G_{x,y;\epsilon}(z)=\mathbf{1}_{B(x,\epsilon)}(z)-\mathbf{1}_{B(y,\epsilon)}(z).$$
We divide the proof into three steps.

\bigskip
\noindent{\it\large Step 1}\,:\ \ 
For any $\epsilon>0$, the following holds:
\begin{align}
&\IE\left[L^a\left([s,t],B(x,\epsilon)\right)-L^a\left([s,t],B( y, \epsilon)\right)\right]^n \nonumber
\\
=&n! \IE  \int_{s\leq u_1\leq\cdots\leq u_n\leq t} G_{x,y;\epsilon}(X_{u_1})\cdots G_{x,y;\epsilon}(X_{u_n})\,   du_1\cdots du_n \nonumber
\\
=&n!  \int_{s\leq u_1\leq\cdots\leq u_n\leq t}du_1\cdots du_n\,  \IE\bigg[\mathbf{1}_{B(x, \epsilon)}(X_{u_1})\mathbf{1}_{B(x, \epsilon)}(X_{u_2})-\mathbf{1}_{B(y, \epsilon)}(X_{u_1})\mathbf{1}_{B(y, \epsilon)}(X_{u_2}) \nonumber
\\
&\quad\quad\quad -\mathbf{1}_{B(y, \epsilon)}(X_{u_1})\mathbf{1}_{B(x, \epsilon)}(X_{u_2})+\mathbf{1}_{B(y, \epsilon)}(X_{u_1})\mathbf{1}_{B(y, \epsilon)}(X_{u_2})\bigg] G_{x,y;\epsilon}(X_{u_3})\cdots G_{x,y;\epsilon}(X_{u_n}).\nonumber
\end{align}
We first estimate the probability expectation in the above. Note
\begin{align}
&\IE\bigg[\mathbf{1}_{B(x, \epsilon)}(X_{u_1})\mathbf{1}_{B(x, \epsilon)}(X_{u_2})-\mathbf{1}_{B(y, \epsilon)}(X_{u_1})\mathbf{1}_{B(y, \epsilon)}(X_{u_2}) \nonumber
\\
&\quad\quad\quad -\mathbf{1}_{B(y, \epsilon)}(X_{u_1})\mathbf{1}_{B(x, \epsilon)}(X_{u_2})+\mathbf{1}_{B(y, \epsilon)}(X_{u_1})\mathbf{1}_{B(y, \epsilon)}(X_{u_2})\bigg] G_{x,y;\epsilon}(X_{u_3})\cdots G_{x,y;\epsilon}(X_{u_n})\nonumber
\\
=&\int_{\IR^n}\bigg[\mathbf{1}_{B(x, \epsilon)}(z_1)\mathbf{1}_{B(x, \epsilon)}(z_2)-\mathbf{1}_{B(y, \epsilon)}(z_1)\mathbf{1}_{B(y, \epsilon)}(z_2) -\mathbf{1}_{B(y, \epsilon)}(z_1)\mathbf{1}_{B(x, \epsilon)}(z_2)\nonumber\\
&\quad\quad\quad+\mathbf{1}_{B(y, \epsilon)}(z_1)\mathbf{1}_{B(y, \epsilon)}(z_2)\bigg] 
 G_{x,y;\epsilon}(X_{u_3})\cdots G_{x,y;\epsilon}(X_{u_n}) p_{u_1, \cdots, u_n}(z_1, \cdots, z_n)dz_1\cdots dz_n, \label{126253}
\end{align}
where $p_{u_1, \cdots u_n}(z_1, \cdots z_n)$ is the joint density for $(X_{u_1},...,X_{u_n})$. For notation convenience, we denote by $\bar{z}=(z_3,...,z_n)$, and write $p(z_1,z_2,\bar{z})=p_{u_1,...,u_n}(z_1,...,z_n)$. By doing change of variable three times in \eqref{126253} to make all the terms in the square brackets equal to $\mathbf{1}_{B(x, \epsilon)}(z_1)\mathbf{1}_{B(x, \epsilon)}(z_2)$, we can write the last display of \eqref{126253} as 
\begin{align*}
&\int_{\IR^n}\mathbf{1}_{B(x, \epsilon)}(z_1)\mathbf{1}_{B(x, \epsilon)}(z_2)G_{x,y;\epsilon}(z_3)\cdots G_{x,y;\epsilon}(z_n) F_{x,y}(z_1,z_2,\bar{z})dz_1 dz_2d\bar{z}, 
\end{align*}
where
\begin{align*}
&F_{x,y}(z_1,z_2,\bar{z})\\
=&p(z_1, z_2,\bar{z})-p(z_1, z_2-(x-y),\bar{z})-p(z_1-(x-y),z_2,\bar{z})+p(z_1-(x-y),z_2-(x-y),\bar{z}).
\end{align*}
Hence we have for any even integer $n$,
\begin{align}
&\quad \IE\left[L^a\left([s,t],B(x,\epsilon)\right)-L^a\left([s,t],B( y, \epsilon)\right)\right]^n \nonumber 
\\
&\le n! \int_{s\leq u_1\leq\cdots\leq u_n\leq t}du_1\cdots du_n\nonumber
\\ &\quad\quad\ \ \int_{\IR^n}\mathbf{1}_{B(x, \epsilon)}(z_1)\mathbf{1}_{B(x, \epsilon)}\, (z_2)|G_{x,y;\epsilon}(z_3)\cdots G_{x,y;\epsilon}(z_n)|
\cdot|F_{x,y}(z_1,z_2,\bar{z})| dz_1dz_2d\bar{z}.  \label{126404}
\end{align} 

\bigskip
\noindent{\it\large Setp 2}\,:\ \ Applying Taylor's expansion to the first and second order, and by the estimates of derivatives of the  joint density $p(z_1,z_2,\bar{z})$ (Theorem \ref{multi point density}), it can be seen immediately that there exist constants $c_i$ such that the following upper bounds hold.
\begin{equation}\label{126408}
|F_{x,y}(z_1, z_2,\bar{z})| \le c_1|x-y|\left(\frac{1}{(u_3-u_2)^H}+\frac{1}{(u_2-u_1)^H}\right)\frac{1}{(u_n-u_{n-1})^{H}}\cdots \frac{1}{(u_2-u_1)^{H}};
\end{equation}
and
\begin{equation}\label{126411}
|F_{x,y}(z_1, z_2,\bar{z})| \le c_2|x-y|^2\left(\frac{1}{(u_3-u_2)^{2H}}+\frac{1}{(u_2-u_1)^{2H}}\right)\frac{1}{(u_n-u_{n-1})^{H}}\cdots \frac{1}{(u_2-u_1)^{H}}.
\end{equation}
In order to get \eqref{126411} as above, we have  used the elementary inequality  $$(u_3-u_2)(u_2-u_1)\ge \min\{(u_3-u_2)^2, (u_2-u_1)^2\}$$ for the cross term.  

\bigskip
\noindent{\it\large Setp 3}\,:\ \ We split the right hand-side of \eqref{126404} into $J_1+J_2$ where
\begin{align}
 J_1=n! &\int_{\{s\le u_1<\cdots <u_n\le t\}\cap \{u_2-u_1\ge u_3-u_2\}}du_1\cdots du_n\nonumber\\
 &\int_{\IR^n} \mathbf{1}_{B(x, \epsilon)}(z_1)\mathbf{1}_{B(x, \epsilon)}\, (z_2)|G_{x,y;\epsilon}(z_3)\cdots G_{x,y;\epsilon}(z_n)| |F_{x,y}(z_1,z_2,\bar{z})| dz_1dz_2d\bar{z}, \label{126422}
\end{align}
and $J_2$ is defined similarly for  $\{u_2-u_1<u_3-u_2\}$.

In the following, we only show how to bound $J_1$, as $J_2$ can be treated in a similar manner. Below, the positive constants $c_i$ ($1\le i\le 5$) may change from line to line.  From our bound of $F_{x,y}(z_1,z_2,\bar{z})$, we clearly have
\begin{align}
&J_1\nonumber\\
\le& c_1\epsilon^{n} \int_{\left\{{s\le u_1<\cdots< u_n\le t\atop u_2-u_1\ge u_3-u_2}\right\}}\left( \frac{|x-y|}{(u_3-u_2)^H} \wedge \frac{|x-y|^2}{(u_3-u_2)^{2H}}\right)\frac{1}{(u_n-u_{n-1})^{H}}\cdots \frac{1}{(u_2-u_1)^{H}}du_1\cdots du_n\nonumber
\\
\le& c_1\epsilon^{n} \int_{\{s\le u_1<\cdots <u_n\le t\}}\left( \frac{|x-y|}{(u_3-u_2)^H} \wedge \frac{|x-y|^2}{(u_3-u_2)^{2H}}\right)\frac{1}{(u_n-u_{n-1})^{H}}\cdots \frac{1}{(u_2-u_1)^{H}}du_1\cdots du_n\nonumber 
\\
\le& c_1\epsilon^{n}(t-s)^{1-H} \int_{\{s\le u_3<\cdots< u_n\le t\}}du_3\cdots du_n\,\frac{1}{(u_n-u_{n-1})^{H}}\cdots \frac{1}{(u_4-u_3)^{H}} \nonumber
\\
& \qquad\qquad \qquad \qquad \ \  \int_{u_2=s}^{u_3}\left( \frac{|x-y|}{(u_3-u_2)^H} \wedge \frac{|x-y|^2}{(u_3-u_2)^{2H}}\right)\frac{1}{|u_3-u_2|^{H}} du_2. \label{126454}
\end{align}
For the integral with respect to $u_2$ in the last display of \eqref{126454}, we have for any $\delta<\frac{1}{H}-2$,
\begin{align}
&\quad  \int_{u_2=s}^{u_3}\left( \frac{|x-y|}{(u_3-u_2)^H} \wedge \frac{|x-y|^2}{(u_3-u_2)^{2H}}\right)\frac{1}{|u_3-u_2|^{H}}du_2\nonumber 
\\
&=\int_{(u_3-u_2)^H<|x-y|}\left( \frac{|x-y|}{(u_3-u_2)^H} \wedge \frac{|x-y|^2}{(u_3-u_2)^{2H}}\right)\frac{1}{|u_3-u_2|^{H}}du_2 \nonumber
\\
&\qquad \qquad \qquad \quad   +\int_{(u_3-u_2)^H\ge |x-y|}\left( \frac{|x-y|}{(u_3-u_2)^H} \wedge \frac{|x-y|^2}{(u_3-u_2)^{2H}}\right)\frac{1}{|u_3-u_2|^{H}}du_2 \nonumber 
\\
&\le \int_{(u_3-u_2)^H<|x-y|} \frac{|x-y|}{(u_3-u_2)^{2H}} du_2+\int_{(u_3-u_2)^H\ge |x-y|}\frac{|x-y|^2}{(u_3-u_2)^{3H}}du_2  \nonumber
\\
&\le |x-y|\cdot |x-y|^{\frac{1}{H}[1-2H]}+\int_{(u_3-u_2)^H\ge |x-y|}\frac{|x-y|^{(1+\delta)}}{(u_3-u_2)^{(2+\delta)H}}du_2  \nonumber
\\
&\le \left(|x-y|^{\frac{1}{H}-1}+|x-y|^{1+\delta}\right) \le |x-y|^{1+\delta}.\label{126543}
\end{align}
To complete the computation for \eqref{126454}, by following the same line of computation as the proof of Lemma \ref{126456}, we obtain
\begin{equation}\label{126455}
\int_{\{s\le u_3<\cdots u_n\le t\}}\frac{1}{(u_n-u_{n-1})^{dH}}\cdots \frac{1}{(u_4-u_3)^{dH}}du_3\cdots du_n\le c_2(t-s)^{(1-dH)(n-3)+1}.
\end{equation}
Combining \eqref{126543} and \eqref{126455}, we have shown that for any $\delta<\frac{1}{H}-2$, there exists some constant $c_3>0$
\begin{align}
J_1\le c_3\epsilon^{n}|x-y|^{1+\delta}(t-s)^{(1-dH)(n-3)+1}.\label{126545}
\end{align}
Following same line of argument as above, it holds for some $c_4>0$ that
\begin{align}
J_2\le c_4\epsilon^n|x-y|^{1+\delta}(t-s)^{(1-dH)(n-3)+1}.\label{126603}
\end{align}
Therefore we have established the following desired upper bound for \eqref{126404}: 
\begin{equation}\label{126610}
\IE\left[L^a\left([s,t],B(x,\epsilon)\right)-L^a\left([s,t],B( y, \epsilon)\right)\right]^n \le c_5\epsilon^n|x-y|^{1+\delta}(t-s)^{(1-H)(n-3)+1}.
\end{equation}
The proof is thus complete.
\end{proof}

\begin{remark}
Clearly, Lemma \ref{proving-holder-space} is a special case of Lemma \ref{proving-uniform-holder}. The reason that we keep Lemma \ref{proving-holder-space} is because the proof of this simpler result provides a better idea of the  more complex proof of  Lemma \ref{proving-uniform-holder}.
We also remark that by taking $s=a$ in Lemma \ref{proving-uniform-holder}, we have for any $\delta\in(0,\frac{1}{H}-2)$ and even $n\in\mathbb{N}$, there exists a positive constant $C_4$ depending on $T$ and $a$ such that
\begin{align}\label{space holder n}
\IE\left[L^a(t,x)-L^a(t,y)\right]^n\le C_4|x-y|^{1+\delta}.
\end{align}
\end{remark}

\bigskip

Denote by $I_R=[-R,R]\subset\IR$. Now that all the three conditions in Theorem \ref{Berman} have been established in the three lemmas above, the following theorem results from Theorem \ref{Berman} naturally.
\begin{theorem}\label{uniform-holder}
Assume $d=1$ and $H<1/2$. For every $\beta<1-H$,  there exists a version of the local time $L^a(t,x)$ jointly continuous in $(t,x)\in[a,T]\times\IR$. Moreover, there exist random variables $\eta$ and $\Delta$ which are almost surely positive and finite such that
\begin{align*}
|L^a(t+h,x)-L^a(t,x)|\le \Delta|h|^{\beta}, \text{ for all } x,t,h \text{ satisfying } t,t+h\in [a,T], \text{ and }|h|<\eta.
\end{align*}
\end{theorem}
\begin{proof}
Fix $R>0$. Apply Theorem \ref{Berman} to $Y(t,x)=L([a,a+t],x)$ and  by Lemma \ref{126456}, Lemma \ref{proving-uniform-holder} and (\ref{space holder n}), we obtain for each $m\in\mathbb{N}$ a joint continuous version $Y_m(t,x)$ of $Y(t,x)$ on $[0,T]\times I_{mR}$ which is $\beta$-H\"{o}lder continuous in time uniformly in $x\in I_{mR}$, for all $\beta<1-H$. 

Observe that for $m'>m$, $Y_{m'}(t,x)$ is a continuous extension of $Y_m(t,x)$ from $[0,T]\times I_{mR}$ to $[0,T]\times I_{m'R}$.
Therefore, we can define a continuous version of $Y$ on $[0,T]\times\IR$ by
$$Y(t,x)=Y_m(t,x),\quad\text{for}\ (t,x)\in[0,T]\times I_{mR}.$$
The proof is complete by observing that $Y(t,x)=L([a,a+t],x)$ has compact support in $x\in\IR$ (for each fixed $\omega\in\Omega$).
\end{proof}

Recall the definition of pathwise H\"{o}lder exponent of $L^a(t,x)$ in (\ref{def-pointwise-holder}). We are now ready to state the main result of this section.

\begin{theorem}\label{main 1}
Assume $d=1$ and $1/4<H<1/2$. Let $X$ be the solution to equation (\ref{SDE}). The pathwise H\"{o}lder exponent of $L^a(t,x)$ is given by
\begin{equation*}
\alpha(t)=1-H, \quad \mathrm{a.s.}\quad  \text{for all }t\in[a,T].
\end{equation*}
\end{theorem}
\begin{proof}
It follows immediately from Theorem \ref{uniform-holder} that $\alpha_L\ge 1-H$. We show $\alpha_L \le 1-H$ in what follows. Fix any $t>0$. Since the local time vanishes outside the range of $X$, it holds 
\begin{align}
\delta&=\int_{\IR^d} \left(L(t+\delta, x)-L(t,x)\right)dx \nonumber
\\
&\le \sup_{x\in \IR^d}\left(L(t+\delta, x)-L(t, x)\right)\sup_{s,u\in [t, t+\delta]}|X_u-X_s| \nonumber
\\
&\le 2\sup_{x\in \IR^d}\left(L(t+\delta, x)-L(t, x)\right)\sup_{s\in [t, t+\delta]}|X_{t}-X_s|. \label{112346}
\end{align}
\\
\indent Now in order to show $\alpha_L\le 1-H$, it suffices to observe for any $\alpha>0$, by \eqref{112346}, it holds that
\begin{equation*}
 \frac{\sup_{x\in \IR^d}\left(L(t+\delta,x)-L(t, x)\right)}{\delta^\alpha}\ge \frac{\delta^{1-\alpha}}{2\sup_{s\in [t, t+\delta]}|X_{t}-X_s|}.
\end{equation*}
On the other hand, by the construction of rough integrals, we have, for any $\gamma<H$, a finite random variable $G_\gamma$ such that for all $s,t\in[0,T]$,
$$|X_t-X_s|\leq G_\gamma |t-s|^\gamma.$$
Hence for all $\gamma<H$,
$$ \frac{\sup_{x\in \IR^d}\left(L(t+\delta,x)-L(t, x)\right)}{\delta^\alpha}\ge \frac{\delta^{1-\alpha}}{G_\gamma\delta^\gamma}=\frac{\delta^{1-\alpha-\gamma}}{G_\gamma}.$$
For any $\alpha>1-H$, we can always find a $\gamma<H$ such that the right hand side above tends to infinity as $\delta\downarrow0$, which implies $\alpha(t)\leq 1-H$.  The proof is complete.
\end{proof}





\begin{thebibliography}{99}
 


\bibitem{BH} F. Baudoin, M. Hairer: 
A version of H\"ormander's theorem for the fractional Brownian motion. 
{\it Probab. Theory Related Fields} {\bf 139} (2007), no. 3-4, 373--395.



\bibitem{BNOT} F. Baudoin, E. Nualart, C. Ouyang and S. Tindel: On probability laws of solutions to differential systems driven by a fractional Brownian motion. to appear in {\it Ann. Probab.}, (2014).











\bibitem{Berman}S.M. Berman: Gaussian sample functions: uniform dimension and H\"{o}lder conditions nowhere. {\it Nagoya math. J.}, Vol. 46 (1972), 63-86.

\bibitem{Berman2}S.M. Berman: Local nondeterminism and local times of Gaussian processes. {\it Indiana Univ. Math. J.}, {\bf 23}, 1973, 69-94.

\bibitem{BYOZ-book}F. Biagini, Y. Hu, B. {\O}ksendal, T. Zhang: {\it Stochastic Calculus for Fractional Brownian Motion and Applications}. Springer series in Probability and Its Applications, Springer, 2008.

\bibitem{CF} T. Cass, P. Friz: Densities for rough differential equations under H\"{o}rmander's condition. {\it Ann. Math}. {\bf 171}, (2010), 2115--2141. 




\bibitem{CLL} T. Cass, C. Litterer, T. Lyons:
Integrability and tail estimates for Gaussian rough differential equations.
{\it Ann. Probab} {\bf 41} (2013), no. 4, 3026--3050.
















\bibitem{FV-bk}
P. Friz, N. Victoir:
\emph{Multidimensional dimensional processes as rough paths.}
Cambridge University Press (2010).

\bibitem{GH}D. Geman, J. Horowitz: Occupation densities. {\it Ann. Probab.}, 1980, Vol 8, No. 1, 1-67.


\bibitem{Gu}
M. Gubinelli:
Controlling rough paths.
{\it J. Funct. Anal.} {\bf 216}, 86-140 (2004).



\bibitem{H-P}
M. Hairer, N.S. Pillai: Regularity of Laws and Ergodicity of Hypoelliptic SDEs Driven by Rough Paths. {\it Ann. Probab.} {\bf 41} (2013), no. 4, 2544--2598.












\bibitem{LO}S. Lou, C. Ouyang: Fractal dimensions of rough differential equations driven by fractional Brownian motions, \emph{submitted}, (2015).



\bibitem{LQ}
T. Lyons, Z. Qian:
{\it System control and rough paths.}
Oxford University Press (2002).







\bibitem{Nu06}
D. Nualart: \emph{The Malliavin Calculus and Related
Topics.} Probability and its Applications. Springer-Verlag, 2nd
Edition, (2006).












\end{thebibliography}
\end{document}